\documentclass{article}
\usepackage{amsfonts, amssymb, latexsym, amscd}
\usepackage{amsmath, amsthm}
\newtheorem{thm}{Theorem}[section]
\newtheorem{cor}[thm]{Corollary}
\newtheorem{lem}[thm]{Lemma}
\newtheorem{prp}[thm]{Proposition}

\def\1{\bf{1}}
\def\0{\bf{0}}
\def\SG{{\mathcal{G}}}
\def\F{{\mathcal{F}}}

\begin{document}

\title{\bf On signed graphs with at most two eigenvalues unequal to $\pm 1$}

\author{
Willem H. Haemers\thanks{\small{\tt{haemers@uvt.nl}}}
\\
{\small Dept. of Econometrics and O.R., Tilburg University, The Netherlands}
\\[5pt]
Hatice Topcu\thanks{\small{\tt{haticekamittopcu@gmail.com}}}
\\
{\small Dept. of Mathematics, Nev\c{s}ehir Hac{\i} Bekta\c{s} Veli University, Turkey}\\
}
\date{}

\maketitle

\begin{abstract}
\noindent
We present the first steps towards the determination of the signed graphs for which the adjacency matrix has all but at most two
eigenvalues equal to $\pm 1$.
Here we deal with the disconnected, the bipartite and the complete signed graphs.
In addition, we present many examples which cannot be obtained from an unsigned graph or its negative by switching.
\\[5pt]
{Keywords:}~signed graph, graph spectrum, spectral characterization.
AMS subject classification:~05C50.
\end{abstract}

\section{Introduction}

A {\em signed graph} $G^\sigma$ is a graph $G=(V, E)$ together with
a function $\sigma : E \rightarrow \{-1, +1\}$, called the \textit{signature function}.
So, every edge is either positive or negative.
The graph $G$ is called the {\em underlying graph} of $G^\sigma$.
The adjacency matrix $A$ of $G^\sigma$ is obtained from the adjacency matrix of $G$,
by replacing $1$ by $-1$ whenever the corresponding edge is negative.
The signed graph $G^{-\sigma}$ with adjacency matrix $-A$ is called the {\em negative} of $G^\sigma$.
The spectrum of $A$ is also called the spectrum of the signed graph $G^\sigma$.
For a vertex set $X\subset V$, the operation that changes the sign of all edges between $X$ and $V\setminus X$ is called switching.
In terms of the matrix $A$, switching multiplies the rows and columns of $A$ corresponding to $X$ by $-1$.
If a signed graphs can be switched into an isomorphic copy of another signed graph, the two signed graphs are called
{\em switching isomorphic}.
Switching isomorphic signed graphs have similar adjacency matrices and therefore they 
have the same spectrum.

Here we consider signed graphs for which the spectrum has all but at most two eigenvalues equal to $1$ or $-1$.
We define $\SG$ to be the set of signed graphs with this property.
Then $\SG$ is closed under switching, negation, and adding or deleting isolated edges.
The unsigned graphs in $\SG$ have been determined in \cite{CHVW},
and every signed graph switching isomorphic to an unsigned graph in $\SG$ is in $\SG$, and so is its negative.
However, there are many other signed graphs in $\SG$ as we shall see, and the determination of all graphs
in $\SG$ will be more complicated than for the unsigned case.
It is already nontrivial to determine the signed complete graphs in $\SG$ (see Section~\ref{complete}).
Nevertheless we believe that the determination is possible.
Here we deal with the graphs in $\SG$ for which the underlying graph is disconnected, bipartite or complete.
In addition we give many other examples of signed graphs in $\SG$.
We hope to obtain the complete description of signed graphs in $\SG$ in future publications.

We use eigenvalue interlacing and other techniques from linear algebra for which we refer to~\cite{BH}.
Some background on signed graphs can be found in~\cite{BCKW}.
As usual, $J$ is the all-ones matrix and $O$ the all-zeros matrix.
The all-ones and all-zeros vector are denoted by $\1$ and $\0$ respectively.

\section{Disconnected graphs in $\SG$}

We start with a lemma that generalizes a well-known result for unsigned graphs.

\begin{lem}\label{-1}
If a signed graph $G^\sigma$ has smallest eigenvalue at least $-1$, then the underlying graph $G$ is a disjoint union of complete graphs.
\end{lem}

\begin{proof}
For any signature function $\sigma$ the signed path $P_3^\sigma$ has spectrum $\{-\sqrt{2},0,\sqrt{2}\}$.
If $G^\sigma$ has smallest eigenvalue at least $-1$, then by interlacing $P_3^\sigma$ is not an induced signed subgraph of $G^\sigma$,
and therefore every component of $G$ is a complete graph.
\end{proof}

Using this lemma the following results are easily proved.

\begin{prp}
A signed graph $G^\sigma$ has all eigenvalues equal to $\pm 1$ if and only if the underlying graph $G$ is a disjoint union of edges.
\end{prp}

\begin{proof}
The disjoint union of edges clearly has all eigenvalues equal to $\pm 1$ (for any signing).
Conversely, suppose $G^\sigma$ has all eigenvalues equal to $\pm 1$.
Then, by Lemma~\ref{-1}, each component of $G^\sigma$ is a signed complete graph.
If $A$ is the adjacency matrix of such a component of order $m$ (say), then trace$(A^2)=m(m-1)$.
On the other hand, all eigenvalues of $A$ are $\pm 1$, so the sum of the eigenvalues of $A^2$ equals $m$.
This implies that $m=2$.
\end{proof}

\begin{prp}\label{but1}
If $G^\sigma$ is connected and all but one eigenvalues
are equal to $\pm 1$,
then $G^\sigma$ or $G^{-\sigma}$ is switching isomorphic
with the unsigned complete graph $K_n$ with $n\neq 2$.
\end{prp}

\begin{proof}
Obviously $G^\sigma$ or $G^{-\sigma}$ has smallest eigenvalue at least $-1$.
Therefore, by Lemma~\ref{-1}, the underlying graph $G$ is the complete graph $K_n$.
Choose a vertex $v$ and switch such that every edge containing $v$ is positive.
Consider the signed graph $G^{\sigma}_v$ on the remaining vertices.
If $G^{\sigma}_v$ has only positive, or only negative edges, then $G^\sigma$ is switching isomorphic with $K_n$.
So we can assume that $n\geq 4$ and that $G^\sigma_v$ has a vertex $w$ incident with a positive edge $\{w,x\}$ and a negative edge $\{w,y\}$.
If edge $\{x,y\}$ is positive, then $v$, $w$, $x$ and $y$ induce a signed complete graph with one negative edge,
and if edge $\{x,y\}$ is negative, we switch with respect to $v$,
and then $v$, $w$, $x$ and $y$ induce a signed complete graph with one positive edge.
In both cases the subgraph has spectrum $\{\pm 1,\pm\sqrt{5}\}$.
Therefore, by eigenvalue interlacing, $G^\sigma$ has at least two eigenvalues unequal to $\pm 1$. 
\end{proof}

\begin{thm}
If $G^\sigma\in\SG$ is disconnected and $G$ has no isolated edges, then $G^\sigma$ is the disjoint union of two
signed graphs both switching isomorphic with an unsigned complete graph or its negative.
\end{thm}

\begin{proof}
Since $G^\sigma$ is disconnected, $G^\sigma$ has at least two components,
and since $G$ has no isolated edges, each component has an eigenvalue unequal to $\pm 1$.
Therefore $G^\sigma$ has exactly two components, and the result follows from Proposition~\ref{but1}.
\end{proof}

\section{Signed complete graphs in $\SG$}\label{complete}

The adjacency matrix of a signed complete graph is the Seidel matrix of an unsigned graph ($-1$ corresponds to adjacency).
If $S$ and $S'$ are the adjacency matrices of two switching isomorphic signed complete graphs, then the (unsigned)
graphs with Seidel matrices $S$ and $S'$ are called {\em switching equivalent}.

Here we determine all signed complete graphs in $\SG$, or equivalently all Seidel matrices with all but at most two eigenvalues equal
to $\pm 1$.
To achieve this we define $\F$ to be the class of (unsigned) graphs where each member $G$ is a clique extended with some isolated vertices
($G=K_m+\ell K_1$, with $m\geq 1,\ell\geq 0$), or the complement (a {\em complete split graph}).
Note that $\F$ contains all graphs of order at most 3.
Figure~\ref{forb} shows all graphs of order $4$ which are not in $\F$.

\setlength{\unitlength}{2.7pt}
\begin{figure}[h]
\begin{picture}(160,20)(-38,-5)

\put(-30,0){\line(1,0){12}}
\put(-30,12){\line(1,0){12}}
\put(-30,0){\circle*{2}}
\put(-30,12){\circle*{2}}
\put(-18,12){\circle*{2}}
\put(-18,0){\circle*{2}}

\put(-25,-4){$F_0$}

\put(0,0){\line(1,0){12}}
\put(0,0){\line(0,1){12}}
\put(0,0){\circle*{2}}
\put(0,12){\circle*{2}}
\put(12,12){\circle*{2}}
\put(12,0){\circle*{2}}

\put(5,-4){$F_1$}

\put(30,0){\line(1,0){12}}
\put(30,0){\line(0,1){12}}
\put(42,0){\line(0,1){12}}
\put(30,0){\circle*{2}}
\put(30,12){\circle*{2}}
\put(42,12){\circle*{2}}
\put(42,0){\circle*{2}}

\put(35,-4){$F_2$}

\put(60,0){\line(1,1){12}}
\put(60,0){\line(1,0){12}}
\put(60,0){\line(0,1){12}}
\put(60,12){\line(1,0){12}}
\put(60,0){\circle*{2}}
\put(60,12){\circle*{2}}
\put(72,12){\circle*{2}}
\put(72,0){\circle*{2}}

\put(65,-4){$F_3$}

\put(90,0){\line(1,0){12}}
\put(90,0){\line(0,1){12}}
\put(102,0){\line(0,1){12}}
\put(90,12){\line(1,0){12}}
\put(90,0){\circle*{2}}
\put(90,12){\circle*{2}}
\put(102,12){\circle*{2}}
\put(102,0){\circle*{2}}

\put(95,-4){$F_4$}

\end{picture}
\caption{Forbidden induced subgraphs for $\F$}\label{forb}
\end{figure}
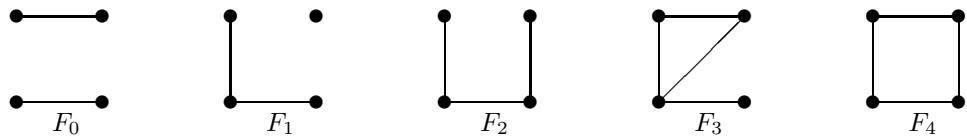

\begin{lem}\label{4}
A graph $G$ is in $\F$ if and only if no induced subgraph of $G$ of order $4$ is isomorphic to a graph in
Figure~\ref{forb}.
\end{lem}

\begin{proof}
It is clear that if $G\in\F$, then $G$ does not contain $F_0$ to $F_4$.
We prove the converse by induction on the number of vertices $n$.
For $n\leq 4$ this is trivial.
Let $G$ be a graph of order $n\geq 5$ that does not contain $F_0$ to $F_4$.
Let $G_v$ be an induced subgraph of $G$ of order $n-1$ obtained by deleting a vertex $v$.
The induction hypothesis gives $G_v\in\F$.
We assume that $G_v=K_m+(n-1-m)K_1$ with $1\leq m \leq n-2$, otherwise we replace $G$ and $G_v$ by their complements
(the set $\{F_0,\ldots,F_4\}$ is closed under taking complements).

First consider the case $m=1$, then $G_v$ has no edges.
Suppose $v$ is adjacent to at least two vertices of $G_v$, and nonadjacent to at least one vertex of $G_v$.
Then $G_v$ contains $F_1$, contradiction.
Therefore $v$ is adjacent to all, one, or no vertices of $G_v$.
But then $G\in\F$.

Next consider the case $2\leq m \leq n-2$.
If $v$ is adjacent to an isolated vertex of $G_v$,
then $G$ contains $F_0$, $F_2$ or $F_3$.
Therefore $v$ is nonadjacent to each isolated vertex of $G_v$.
If $v$ is adjacent to at least one vertex of $K_m$ and also nonadjacent
to at least one vertex of $K_m$, then $G_v$ contains $F_1$.
Therefore $v$ is adjacent to all or no vertices of $K_m$.
Thus we have $G\in\F$.
\end{proof}

The Seidel matrix $S_{m,\ell}$ of $K_m+\ell K_1$ has the form
\[
S_{m,\ell}=
\left[
\begin{array}{cc}
I_m-J & J \\
J & J-I_\ell
\end{array}
\right], \mbox{ and spectrum }
\]
\begin{eqnarray}\label{comspec}
\{-1^{m-1}, 1^{\ell-1}, \textstyle{\frac{1}{2}}(m-\ell \pm\sqrt{m^2+\ell^2+6m\ell-4m-4\ell +4})\}
\end{eqnarray}
(eigenvalue multiplicities are denoted as exponents).
Therefore, if $G \in \F$ then $G$ or its complement has a Seidel matrix with the above spectrum.
Hence all but at most two Seidel eigenvalues of $G$ are equal to $\pm 1$.
The next theorem shows that also the converse is true.
For future references we use a slightly stronger statement.

\begin{thm}\label{Seidel}
Assume $G$ is a graph for which the Seidel matrix has all but at most two eigenvalues in the interval $[-1,1]$.
Then $G$ is switching equivalent to a graph in $\F$.
\end{thm}

\begin{proof}
We may assume that $G$ has an isolated vertex $v$.
Let $G_v$ be the graph induced by the remaining vertices.
We will use Lemma~\ref{4} to prove that $G_v\in\F$.
Let $F$ be an induced subgraph of $G_v$ of order $4$.
Then $F+K_1$ is an induced subgraph of $G$, and by interlacing, its Seidel matrix has at least three eigenvalues in the interval $[-1,1]$.
The Seidel matrix of $F_i+K_1$ has spectrum $\{-\sqrt{5}{\,}^2,0,\sqrt{5}{\,}^2\}$ if $i=2$, and 
$\pm \{ -1^2, 3, \frac{1}{2}(-1\pm\sqrt{17}) \}$ if $i=0,1,3$ or $4$ (see~\cite{vLS}).
Therefore $F\neq F_0,\ldots,F_4$, hence $G_v\in\F$.
This implies that either $G\in\F$, or $G$ will become a member of $\F$ after switching with respect to $v$.
\end{proof}
Apparently, the adjacency matrix $A$ of a signed complete graph $G^\sigma$ in $\SG$ is equal to the Seidel matrix of a graph in $\F$.
So $A$ is switching equivalent with the Seidel matrix $\pm S_{m,\ell}$ (defined above).
Note that $S_{m,\ell}$ is switching equivalent with $-S_{\ell,m}$, therefore:

\begin{cor}\label{corS}
If $G^\sigma$ is a signed complete graph in $\SG$, then $G^\sigma$ is switching isomorphic with a signed graph with adjacency matrix $S_{m.\ell}$. 
\end{cor}

\section{Bipartite graphs in $\SG$}\label{bipartite}

Suppose $G^\sigma$ is a connected bipartite signed graph in $\SG $ with adjacency matrix $A$. Then we may assume
\[
A=\left[
\begin{array}{cc}
O & N\\N^\top & O
\end{array}
\right],\ {\rm and}
\ A^2=\left[
\begin{array}{cc}
NN^\top & O \\ O & N^\top\! N
\end{array}
\right].
\]
Note that for a bipartite signed graph, $G^\sigma$ and $G^{-\sigma}$ are switching isomorphic.
In particular, $A$ and $-A$ have the same spectrum.
If $A$ has an eigenvalue $0$, then the smallest eigenvalue equals $-1$, and Lemma~\ref{-1} implies $G=K_2$.
So $A$ is nonsingular, and therefore $N$ is a square matrix.
Moreover, if $G\neq K_2$ then $A^2$ has all but two eigenvalues equal to $1$, which implies that rank$(NN^\top\! -I)=1$.
Van~Dam and Spence~\cite{vDS} have proved that for a connected bipartite (unsigned) graph $G\in\SG$ ($G\neq K_2$)
the matrix $N$ is one of the following:
\begin{eqnarray}\label{N}
\ N=
\left[
\begin{array}{cc}
J-I_3 & J \\ O &J-I_3
\end{array}
\right],
\ {\rm or}
\ N=
\left[
\begin{array}{cc}
1 & \1^\top \\ \1 & I_4
\end{array}
\right],
\ {\rm or}
\ N=J-I_m\ (m\geq 3),
\end{eqnarray}
and the spectrum of $G$ is $\{-1^5, 1^5, \pm 4\}$, $\{-1^4, 1^4, \pm 3\}$, and $\{-1^{m-1}, 1^{m-1}, \pm(m-1)\}$, respectively.
When a signed graph $G^\sigma$ is switching isomorphic with one of the above graphs, then
$G^\sigma$ is a connected bipartite signed graphs in $\SG$.
But also the converse is true.

\begin{thm}
If $G^\sigma \in\SG$ is connected and bipartite then $G^\sigma$ is switching isomorphic with one of the
unsigned connected bipartite graphs in $\SG$.
\end{thm}

\begin{proof}
We follow the steps for the unsigned classification, given by Van~Dam and Spence~\cite{vDS}.
The case $G=K_2$ is trivial.
If $G\neq K_2$ then $N$ has at least two rows. 
Suppose $r_1$ and $r_2$ are two distinct rows of $N$ with weights $k_1$ and $k_2$ respectively.
Without loss of generality, we assume that $r_1$ has no negative entries, and that no row of $N$ has weight smaller than $k_1$.
Then $NN^\top\! -I$ has the following principal submatrix
\[
B=
\left[
\begin{array}{cc}
k_1-1 & x \\ x & k_2-1
\end{array}
\right],\ {\rm where}\ x=r_1 {r_2}^\top.
\]
We may assume that $x\geq 0$ (otherwise we replace $r_2$ by $-r_2$).
Since rank$(NN^\top -I)= 1$, $B$ is singular, and therefore
\begin{eqnarray}\label{sing}
(k_1-1)(k_2-1)=x^2.
\end{eqnarray}
One solution is that $x=0$ and $k_1=1$.
But then $G^\sigma$ is disconnected.
It is clear that $x\leq k_1$.
If $x=k_1$ then equation~\ref{sing} gives $(x-1)(k_2-1)=x^2$,
hence $k_2=x+2+1/(x-1)$ which implies $x=k_1=2$, and $k_2=5$.

First assume $k_1 \geq 3$.
Then $x\neq k_1$, so $x\leq k_1-1$, and equation~(\ref{sing}) implies that $x=k_1-1=k_2-1$.
Therefore all rows of $N$ have weight $k_1$, and we can permute and switch in $N$ such that $r_1$ and $r_2$ look as follows:
\[
\left[
\begin{array}{c}
r_1 \\ r_2
\end{array}
\right]
=
\left[
\begin{array}{cccccc}
1\ \ldots\ 1 & 1 & 1 & 0 & 0 & 0\ \ldots\ 0 \\
1\ \ldots\ 1 & 1 & 0 & 1 & 0 & 0\ \ldots\ 0
\end{array}
\right].
\]
Let $r_3$ be a row of $N$ different from $r_1$ and $r_2$.
($N$ has at least three rows because $k_1\geq 3$ and $N$ is square matrix).
Then $r_1 r_3^\top=k_1-1$ (we replace $r_3$ by $-r_3$ if $r_1 r_3^\top$ is negative)
and $r_2 r_3^\top=\pm (k_1-1)$.
This leads to the following possibilities for $r_3$ (here we use that $k_1\geq 3$):
\begin{eqnarray*}
r_3 &=&
\left[
\begin{array}{cccccc}
1\ \ldots\ 1 & 1 & 0 & 0 & \!\!\pm 1\!\! & 0\ \ldots\ 0
\end{array}
\right],\ {\rm or}
\\
r_3 &=&
\left[
\begin{array}{cccccc}
1\ \ldots\ 1 & 0 & 1 & 1 & \, 0 & 0\ \ldots\ 0
\end{array}
\right].
\end{eqnarray*}
The first case cannot be completed to a square matrix, and the second case leads to $N=J-I_m$ with $m=k_1+1$.

Next we consider the case $k_1=2$.
Now each row of $N$ has weight $2$ or $5$.
If all rows have weight 2, then any two distinct rows have inner product $\pm 1$, and we find just two nonequivalent possibilities:
\[
N=\left[\begin{array}{ccc}
1 &  1 &  0 \\
1 &  0 &  1 \\
0 &  1 &  1
\end{array}\right],\ {\rm or}
\ N={\left[\begin{array}{ccr}
1 &  1 &  0 \\
1 &  0 &  1 \\
0 &  1 & \!\!-1
\end{array}\right]}.
\]
However the second case does not occur, since then rank$(NN^\top\!-I) = 3$.

So some rows have weight 2, and some have weight 5.
We have seen that a row of weight $2$ and a row of wight 5 have inner product $\pm 2$,
that two rows of weight $5$ have inner product $\pm 4$,
and that two rows of weight $2$ have inner product $\pm 1$.
If there is just one row of weight $5$ and the other ones have weight $2$,
then there is up to equivalence just one possibility: the second matrix in~(\ref{N}).
With two or four rows of weight $5$ there is no solution, but with three rows of weight $5$
we find the first matrix in~(\ref{N}).
\end{proof}

\section{Other signed graphs in $\SG$}\label{more}

Here we present some signed graphs in $\SG$ which are not described in one of the previous sections.
The reverse identity matrix of order $m$ is denoted by $R_m$ (that is, $R_m(i,j)=1$ if $i+j=m+1$,
and $R_m(i,j)=0$ otherwise).
Note that all eigenvalues of $R_m$ are equal to $\pm 1$.

\begin{thm}
The following matrices represent signed graphs in $\SG$.
\\[5pt]
$
A_1=\left[
\begin{array}{cc}
J\!-\!I_m & J \\ \!J & \!-R_{2\ell}
\end{array}
\right]
$
($m,\ell\geq 1$) with spectrum $\{-1^{\ell+m-2},1^\ell,\frac{1}{2}(m-2\pm\sqrt{m(m+8\ell)}\,)\}$,
\\[5pt]
$
A_2=\left[
\begin{array}{cc}
R_{2m} & J \\ J & -R_{2\ell}
\end{array}
\right]
$
($m,\ell\geq 1$) with spectrum $\{-1^{m+\ell-1},1^{m+\ell-1},\pm\sqrt{1+4m\ell}\}$,
\\[5pt]
$
A_3=\left[
\begin{array}{rrcr}
\!J\!-\!I_m\! & \1     & \1  & O\ \\
\ \1^\top\    & 0      &  1  & -\1^\top \\
\ \1^\top\    & 1      &  0  &  \1^\top \\
O\ \          & \!-\1  & \1  & \!I_\ell-\!J\!
\end{array}
\right]
$
($m,\ell\geq 1$) with spectrum $\{-1^m,1^\ell,-\ell-1,m+1\}$,
\\[5pt]
$
A_4=\left[
\begin{array}{rccr}
\!J\!-\!I_m\! & J   & \ O   &  J\ \\
J\ \          & R_2 & \ O   &  O\ \\
O\ \          & O   & -R_2  & -J\ \\
J\ \          & O   & -J    & \!I_\ell-\!J\!
\end{array}
\right]
$\hspace{-3pt}
\begin{tabular}{ll}
($m,\ell\geq 1$) with spectrum $\{-1^{m+1},1^{\ell+1}$,\\[3pt]
$\frac{1}{2}(m-\ell\pm\sqrt{m^2+\ell^2+6m\ell+4m+4\ell+4})\}$.
\end{tabular}

\end{thm}
\begin{proof}
We use the method described in Lemma~1 of \cite{CHVW} (see also Section~2.3 of \cite{BH}).
For $i=1,\ldots,4$ the given block structure of $A_i$ corresponds to an equitable partition,
and therefore $A_i$ has two kinds of eigenvalues,
the ones with an eigenvector in the space $\cal V$ spanned by the characteristic vectors of the partition,
and the ones with an eigenvector orthogonal to $\cal V$.
The eigenvalues of the first kind are the eigenvalues of the quotient matrices
\[
\left[
\begin{array}{cc}
m-1 & 2\ell \\
m & -1
\end{array}
\right],\
\left[
\begin{array}{cc}
1 & 2\ell \\ 2m & -1
\end{array}
\right],\
\left[
\begin{array}{crrr}
m\!-\!1 & \ 1 & 1 & 0\ \\
m\      & 0   & 1 & -\ell\ \\
m\      & 1   & 0 & \ell\ \\
0\      & -1  & 1 & \ 1\!-\!\ell
\end{array}
\right]
\mbox{and}\
\left[
\begin{array}{ccrr}
m\!-\!1 & 2 &  0 & \ell\ \\
m      & 1 &  0 & 0\ \\
0      & 0 & -1 & -\ell\ \\
m      & 0 & -2 & \ 1\!-\!\ell
\end{array}
\right].
\]
They have spectra $\{\frac{1}{2}(m-2\pm\sqrt{m(m+8\ell)}\,)\}$, $\{\pm\sqrt{1+4m\ell}\}$,
$\{\pm 1,-\ell-1,m+1\}$, and $\{\pm 1,\frac{1}{2}(m-\ell\pm\sqrt{m^2+\ell^2+6m\ell+4m+4\ell+4})\}$, respectively.
The eigenvalues of the second kind remain unchanged if we add $\pm J$ to some of the blocks,
so they are also eigenvalues of
\[
\left[
\begin{array}{cc}
-I_m & O \\
O & -R_{2\ell}
\end{array}
\right],\
\left[
\begin{array}{cc}
R_{2m} & O \\ O & -R_{2\ell}
\end{array}
\right],\
\left[
\begin{array}{rrrc}
-I_m & \0 & \0 & O\ \\
\0^\top & 0 & 1 & \0^\top \\
\0^\top & 1 & 0 & \0^\top \\
O\ & \0 & \0 & I_\ell
\end{array}
\right] \mbox{and}\
\left[
\begin{array}{cccc}
-I_m & O &  O & O\ \\
O    & R_2 &  O & O \\
O    & O & I_2 & O \\
O    & O &  O & I_\ell
\end{array}
\right],
\]
which are clearly all equal to $\pm 1$.
\end{proof}

Obviously signed graph obtained from one of the above examples by switching or taking the negative are also in $\SG$.
Only the signed graph represented by $A_1$ with $m=1$ is switching isomorphic with its underlying graph (which is the Friendship graph).
All the other ones are not switching isomorphic with an unsigned graph or its negative.
For $A_1$ and $A_2$ the underlying graphs are also in $\SG$ (see \cite{CHVW}, Theorem~1),
but for $A_3$ and $A_4$ this is not the case.
Note that $A_1$ with $m=2$ is equal to $A_2$ with $m=1$.

Zoran Stani\'c~\cite{S} has generated all nonequivalent connected signed graphs on at most eight vertices, and computed their spectra.
By checking Stani\'c's list we found that the numbers of nonequivalent connected graphs in $\SG$ of orders $1$ to $8$ are equal to
$1,1,2,4,8,14,20,29$, respectively.
It turns out that only four signed graphs on at most eight vertices haven't been described above.
These four sporadic signed graphs are the two graphs in Figure~\ref{sporadic} and their negatives.

\setlength{\unitlength}{3pt}
\begin{figure}[h]
\begin{picture}(36,15)(-20,-10)
\put(0,0){\line(0,1){10}}
\put(0,0){\line(1,0){20}}
\put(10,0){\line(0,1){10}}
\put(20,0){\line(0,1){10}}
\put(0,10){\line(1,0){20}}

\put(20,0){\line(2,1){10}}
\put(20,10){\line(2,-1){10}}

\put(10,0){\line(1,1){10}}
\put(10,0){\line(-1,1){10}}
\put(10,10){\line(-1,-1){10}}
\put(10,10){\line(1,-1){10}}

\put(-3.5,4){$-$}

\put(0,0){\circle*{2}}
\put(0,10){\circle*{2}}
\put(10,0){\circle*{2}}
\put(10,10){\circle*{2}}
\put(20,0){\circle*{2}}
\put(20,10){\circle*{2}}
\put(30,5){\circle*{2}}

\put(-8,-10){spectrum $\{-1^3,1^2,\frac{1}{2}(1\pm\sqrt{41})\}$}
\end{picture}
\begin{picture}(32,25)(-40,-10)
\put(10,0){\line(0,1){10}}
\put(10,0){\line(2,1){10}}
\put(10,0){\line(-2,1){10}}
\put(10,10){\line(-2,-1){10}}
\put(10,10){\line(2,-1){10}}
\put(30,10){\line(-1,1){5}}
\put(30,0){\line(-1,-1){5}}

\put(20,5){\line(1,2){5}}
\put(20,5){\line(2,1){10}}
\put(20,5){\line(1,-2){5}}
\put(20,5){\line(2,-1){10}}

\put(6.5,4){$-$}

\put(10,0){\circle*{2}}
\put(10,10){\circle*{2}}
\put(0,5){\circle*{2}}
\put(20,5){\circle*{2}}
\put(30,0){\circle*{2}}
\put(30,10){\circle*{2}}
\put(25,-5){\circle*{2}}
\put(25,15){\circle*{2}}

\put(-2,-10){spectrum $\{-1^3,1^3,\pm2\sqrt{2}\}$}
\end{picture}
\caption{Sporadic signed graphs in $\SG$}\label{sporadic}
\end{figure}
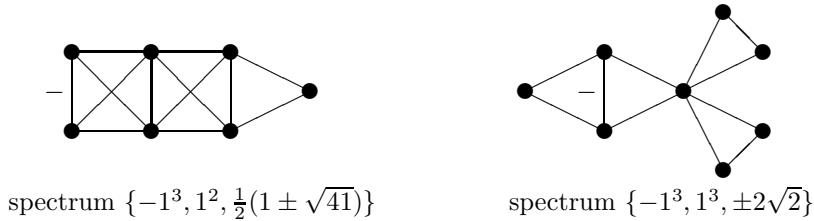

\section{Concluding remarks}
A complete determination of the signed graphs in $\SG$ will make it possible to decide which signed graphs in $\SG$ are determined
by their spectrum (up to switching), and which ones are not.
Nevertheless, the results if this paper already lead to some interesting conclusions.

\begin{thm}\label{compDS}
A complete signed graph in $\SG$ is determined by its spectrum (up to switching).
\end{thm}

\begin{proof}
If $G^\sigma$ is a signed graph of order $n$ with adjacency matrix $A$ and eigenvalues $\lambda_1,\ldots,\lambda_n$,
then we easily have that $\sum_i \lambda_i^2=\mbox{ trace}(A^2)\leq n(n-1)$ with equality if and only if $G^\sigma$ is complete.
If $G^\sigma$ has the spectrum given in formula~(\ref{comspec}) of Section~\ref{complete}, then $\ell$ and $m$ are determined
and $\sum_i\lambda_i^2 = n(n-1)$, so $G^\sigma$ is a complete signed graph in $\SG$ which, by Corollary~\ref{complete},
has an adjacency matrix switching equivalent with $S_{m,\ell}$.
\end{proof}

An unsigned graph is bipartite if and only if the spectrum is symmetric (that is, the spectrum is invariant under multiplication by $-1$).
This however, doesn't hold for signed graphs.
Here we found many counter examples.
In fact, for every bipartite signed graph $G^\sigma\in\SG$ there is a signed graph presented in Section~\ref{more}
with adjacency matrix $A_3$ which, when extended with some isolated edges, has the same spectrum as $G^\sigma$.
So, in contrast to Theorem~\ref{compDS}, none of the bipartite signed graphs in $\SG$ is determined by the spectrum.

Many graphs in $\SG$ have a symmetric spectrum.
This includes the  bipartite ones, some of the signed complete graphs (the case $m=\ell$), and several signed graphs
from Section~\ref{more}.
Signed graphs which are switching isomorphic with their negatives are called {\em sign-symmetric}.
We already observed that bipartite signed graphs are sign-symmetric, and clearly sign-symmetric graphs have a symmetric spectrum.
Signed graphs with symmetric spectrum which are not sign-symmetric are of special interest, see~\cite{BCKW} and \cite{GHMP}.
Here we find more examples.
The second signed graph of Figure~\ref{sporadic} is not sign-symmetric, and neither are
the ones constructed in Section~\ref{more} with adjacency matrix $A_2$ and $m\neq\ell$.
When $1+4m\ell$ is a square, then for each of the latter examples there exists a bipartite signed graph with the same spectrum
(we may have to extend one of the signed graphs with some isolated edges).
We can even find such a pair of connected signed graphs.
Indeed, the bipartite signed graph of order 16 from (\ref{N}) in Section~\ref{bipartite} (take $m=8$ in the third matrix)
has the same spectrum as the one represented by $A_2$ with $m=2$ and $\ell=6$, which is not sign-symmetric.
\\

\noindent
{\bf Acknowledgement}
We thank Hakan K\"{u}\c{c}\"{u}k for checking which graphs from the list of Stani\'c belong to $\SG$.

\end{document}